\documentclass{article}

\usepackage[english]{babel}

\usepackage[letterpaper,top=2cm,bottom=2cm,left=3cm,right=3cm,marginparwidth=1.75cm]{geometry}

\usepackage{amsmath}
\usepackage{amssymb}
\usepackage{dsfont}
\usepackage{graphicx}
\usepackage[colorlinks=true, allcolors=blue]{hyperref}
\usepackage{amsthm}

\newtheorem{theorem}{Theorem}[section]
\newtheorem{lemma}[theorem]{Lemma}
\newtheorem{corollary}[theorem]{Corollary}
\newtheorem{conjecture}[theorem]{Conjecture}

\newtheorem{notation}[theorem]{Notation}
\newtheorem{example}[theorem]{Example}
\newtheorem{definition}[theorem]{Definition}

\title{$\sigma$-Sets and $\sigma$-Antisets}
\author{  {\bf   Ivan Gatica and Alfonso Bustamante }
	\\ \\
	Institute of Mathematics, Pontifical Catholic University of Valparaiso,\\
	Valparaiso, Chile\\
	{\small ivan.gatica.a@mail.pucv.cl}\\    
    Department of Informatics, Technological University of Chile,\\ Apoquindo 7282, Santiago, Chile,\\
	{\small alfonso.bustamante02@inacapmail.cl}\\  	
}

\begin{document}
\maketitle

\begin{abstract}
 In this paper we present a brief study of the $\sigma$-set-$\sigma$-antiset duality that occurs in $\sigma$-set theory and we also present the development of the integer space $3^{A}=\left\langle 2^{A}, 2^{A^{-}} \right\rangle$ for the cardinals $|A|=2,3$ together with its algebraic properties. In this article, we also develop a presentation of some of the properties of fusion of $\sigma$-sets and finally we present the development and definition of a type of equations of one $\sigma$-set variable.
\end{abstract}

\section{$\sigma$-Sets and $\sigma$-Antisets}

As we have seen in \cite{Gatica}, an $\sigma$-antiset is defined as follows:

\begin{definition} Let $A$ be a $\sigma$-set, then $B$ is said to be the $\sigma$-antiset of $A$ if and only if $A\oplus B=\emptyset$, where $\oplus$ is the fusion of $\sigma$-sets.
\end{definition}
We must observe that given the definition of the fusion operator $\oplus$ in \cite{Gatica} it is clear that it is commutative and therefore if $B$ is an $\sigma$-antiset of $A$, then it will be necessary that $A$ is also the $\sigma$-antiset of $B$. On the other hand, following the Blizard notation, \cite{Blizard} p. 347, we will denote $B$ the $\sigma$-antiset of $A$ as $B=A^{-}$, in this way we will have $A=(A^{-})^{-}$.

Continuing with the development of the $\sigma$-sets we have constructed three primary $\sigma$-sets, which are:

\begin{center}
\begin{tabular}{|l|l|}
\hline
Natural Numbers     & $\mathds{N}=\{1,2,3,4,5,6,7,8,9,10,\ldots\}$ \\
\hline
$0$-Natural Numbers & $\mathds{N}^{0}=\{1_{0},2_{0},3_{0},4_{0},5_{0},6_{0},7_{0},8_{0},9_{0},10_{0},\ldots\}$ \\
\hline
Antinatural Numbers & $\mathds{N}^{-}=\{1^{\ast},2^{\ast},3^{\ast},4^{\ast},5^{\ast},6^{\ast},7^{\ast},8^{\ast},9^{\ast},10^{\ast},\ldots\}$\\
\hline
\end{tabular}
\end{center}
where $1=\{\alpha\}$, $1_{0}=\{\emptyset\}$ and $1^{\ast}=\{\omega\}$, we must clarify that we have changed the letter $\beta$ for the letter $\omega$ for symmetry reasons, we must also remember that:
$$\ldots\in\alpha_{-2}\in\alpha_{-1}\in\alpha\in\alpha_{1}\in\alpha_{2}\ldots$$
and
$$\ldots\in\omega_{-2}\in\omega_{-1}\in\omega\in\omega_{1}\in\omega_{2}\in\ldots$$
where both $\epsilon$-chains have the linear $\epsilon$-root property and are totally different, i.e. they do not have a link-intersection. These definitions can be found in \cite{Gatica} Definition 3.13, 3.14 and 3.16.

On the other hand, we must remember the definition of the space generated by two $\sigma$-sets $A$ and $B$ which is:
\begin{definition}\label{Df CG}
    Let $A$ and $B$ be two $\sigma$-sets. The \textbf{Generated space by $A$ and $B$ } is given by 
    $$ \left\langle 2^{A},2^{B}\right\rangle=\{x\oplus y : x\in 2^{A}\wedge y\in 2^{B}\},$$
    where $\oplus$ is the fusion operator.
\end{definition}
Let us recall a few things about the fusion operator $\oplus$. In this brief analysis, we must observe that given $x$, $y$ two $\sigma$-sets, if $\{x\}\cup\{y\}=\emptyset$ then it will be said that $y$ is the antielement of $x$ and $x$ the antielement of $y$,
where the union of pairs $\cup$ axiomatized within the theory of $\sigma$-sets is used, in particular in the completion axioms A and B, which we will call annihilation axioms from now on. 
\begin{notation}
    Let $x$ be an element of some $\sigma$-set, then we will denote by $x^{\ast}$ the anti-element of $x$, if it exists.
\end{notation}
Now we move on to define the new operations with $\sigma$-sets which will help us define the fusion of $\sigma$-sets $\oplus$.
\begin{definition}
    Let $A$ and $B$ be two $\sigma$-sets, then we define the $\ast$-intersection of $A$ with $B$ by
    $$A\widehat{\cap} B =\{x\in A : x^{\ast}\in B \}.$$  
\end{definition}

\begin{example}
Let $A=\{1,2,3^{\ast},4\}$ and $B=\{2,3,4^{\ast}\}$ be two $\sigma$-sets, then we have that:
$$A\widehat{\cap}B=\{3^{\ast},4\}$$
and
$$B\widehat{\cap}A=\{3,4^{\ast}\},$$
it is clear that the $\ast$-intersection operator is not commutative.
\end{example}

\begin{theorem}\label{T IA}
 Let $A$ be a $\sigma$-set, then $A\widehat{\cap} A=\emptyset$.
\end{theorem}
\begin{proof}
    Let $A$ be a $\sigma$-set, by definition we will have that
$$A\widehat{\cap}A=\{x\in A: x^{\ast}\in A\}.$$
Suppose now that $A\widehat{\cap}A\neq\emptyset $, then there exists an $x\in A$ such that $x^{\ast}\in A$, therefore we will have that $x,x^{\ast}\in A$, which is a contradiction with Theorem 3.39 (Exclusion of inverses) from \cite{Gatica}, so if $A$ is a $\sigma$-set then
$$A\widehat{\cap}A=\emptyset.$$ 
           
\end{proof}

\begin{example}
    Let $A=\{1,2,3^{*},4\}$, then 
    $$A\widehat{\cap}A=\{1,2,3^{*},4\}\widehat{\cap}\{1,2,3^{*},4\},$$
    $$A\widehat{\cap}A=\{x\in\{1,2,3^{*},4\}: x^{*}\in\{1,2,3^{*},4\}\}, $$
    $$A\widehat{\cap}A=\emptyset.$$

\end{example}

Regarding Theorem \ref{T IA}, we can observe that given a $\sigma$-set $A$, the $\sigma$-set theory does not allow the coexistence of a $\sigma$-element $x$ and its $\sigma$-antielement in the same $\sigma$-set $A$, and this is because $A$ is a $\sigma$-set. However, since $\sigma$-set theory is a $\sigma$-class theory, one can find the $\sigma$-elements together with the $\sigma$-antielements coexisting without problems in what we call the proper $\sigma$-class, in this way one will have that $\{x,x^{\ast}\}$ is a proper $\sigma$-class and not a $\sigma$-set.

\begin{theorem}\label{T IV}
 Let $A$ be a $\sigma$-set, then $A\widehat{\cap}\emptyset=\emptyset$ and $\emptyset\widehat{\cap}A=\emptyset$.
\end{theorem}
\begin{proof}
    Let $A$ be a $\sigma$-set, by definition we will have that
$$A\widehat{\cap}\emptyset=\{x\in A: x^{\ast}\in \emptyset\}.$$
Now suppose that $A\widehat{\cap}\emptyset\neq\emptyset $, then there exists an $x\in A$ such that $x^{\ast}\in\emptyset$, which is a contradiction, hence $A\widehat{\cap}\emptyset=\emptyset $.
On the other hand, $\emptyset\widehat{\cap}A\subseteq\emptyset $ thus we will have to $\emptyset\widehat{\cap}A=\emptyset$.
\end{proof}

On the other hand, we will define the $\ast$-difference between $\sigma$-sets, a fundamental operation to be able to define the fusion between $\sigma$-sets.

\begin{definition}
    Let $A$ and $B$ be two $\sigma$-sets, then we define the $\ast$-difference between $A$ y $B$ by 
    $$ A\divideontimes B =A - (A\widehat{\cap} B),$$
    where $A-B=\{x\in A : x\notin B\}.$
\end{definition}

\begin{example}
Let $A=\{1,2,3^{\ast},4\}$ and $B=\{2,3,4^{\ast}\}$, then we have that:
$$A\widehat{\cap}B=\{3^{\ast},4\},$$
therefore
$$ A\divideontimes B =A - (A\widehat{\cap} B)=\{1,2,3^{\ast},4\}-\{3^{\ast},4\}=\{1,2\}$$
$$ A\divideontimes B =\{1,2\}.$$
We also have to
$$B\widehat{\cap}A=\{3,4^{\ast}\}$$
therefore 
$$ B\divideontimes A =B - (B\widehat{\cap} A)=\{2,3,4^{\ast}\}-\{3,4^{\ast}\}=\{2\}$$
$$ B\divideontimes A =\{2\}.$$
\end{example}

\begin{corollary}\label{C DA}
    Let $A$ be a $\sigma$-set. Then $A\divideontimes A=A$.
\end{corollary}
\begin{proof}
Let $A$ be a $\sigma$-set, then by Theorem \ref{T IA} we will have that $A\widehat{\cap}A=\emptyset$
therefore
$$A\divideontimes A=A - (A\widehat{\cap} A)=A-\emptyset =A.$$ 

\end{proof}

\begin{corollary}\label{C DV}
    Let $A$ be a $\sigma$-set. Then $A\divideontimes\emptyset=A$ and $\emptyset\divideontimes A=\emptyset$.
\end{corollary}
\begin{proof}
Let $A$ be a $\sigma$-set, then by Theorem \ref{T IV} we will have that $A\widehat{\cap}\emptyset=\emptyset\widehat{\cap}A=\emptyset$
therefore  
$$A\divideontimes\emptyset=A - (A\widehat{\cap} \emptyset)=A-\emptyset =A$$ 
and 
$$\emptyset\divideontimes A=\emptyset - (\emptyset\widehat{\cap} A)=\emptyset-\emptyset =\emptyset.$$

\end{proof}

Now after defining the $\ast$-intersection and the $\ast$-difference we can define the fusion of $\sigma$-sets as follows:
\begin{definition}
    Let $A$ and $B$ be two $\sigma$-sets, then we define the fusion of $A$ and $B$ by 
    $$ A\oplus B =\{ x : x\in A\divideontimes B \vee x\in B\divideontimes A\}.$$
\end{definition}
It is clear that the fusion of $\sigma$-sets is commutative by definition. Now, let us show an example
\begin{example}
Let $A=\{1,2,3^{\ast},4\}$ y $B=\{2,3,4^{\ast}\}$, then we have that:
$$ A\oplus B =\{ x : x\in A\divideontimes B \vee x\in B\divideontimes A\},$$
$$A\oplus B =\{x : x\in\{1,2\}\vee x\in\{2\}  \}, $$
$$A\oplus B=\{1,2\}, $$
therefore we have that 
$$\{1,2,3^{\ast},4\}\oplus\{2,3,4^{\ast}\}=\{2,3,4^{\ast}\}\oplus \{1,2,3^{\ast},4\} =\{1,2\}.$$

\end{example}

\begin{corollary}\label{C SA}
    Let $A$ be a $\sigma$-set, then $A\oplus A=A$.
\end{corollary}
\begin{proof}
Let $A$ be a $\sigma$-set, by definition we have that,
    $$A\oplus A=\{x: x\in A\divideontimes A \vee x\in A\divideontimes A\}.$$
    Now by corollary \ref{C DA}, we have that  
    $$A\oplus A=\{x: x\in A \vee x\in A\},$$
     $$A\oplus A=\{x: x\in A\},$$
     therefore it is clear that $A\subset A\oplus A$ and that $A\oplus A\subset A$, therefore $A\oplus A=A$.

\end{proof}

\begin{corollary}\label{C SV}
    Let $A$ be a $\sigma$-set, then $A\oplus\emptyset=\emptyset\oplus A=A$.
\end{corollary}
\begin{proof}
   First we will show that $A\oplus\emptyset=A$. By definition we will have that,
$$A\oplus\emptyset=\{x: x\in A\divideontimes\emptyset \vee x\in\emptyset\divideontimes A\}.$$
Now by the corollary \ref{C DV}, we will have that
$$A\oplus\emptyset=\{x: x\in A \vee x\in\emptyset\},$$
$$A\oplus\emptyset=\{x: x\in A\},$$
from this it is clear that $A\subset A\oplus\emptyset$ and that $A\oplus\emptyset\subset A$, in this way $A\oplus\emptyset=A$.

Second, we will show that $\emptyset\oplus A=A$. By definition we will have that,
$$\emptyset\oplus A =\{x: x\in\emptyset\divideontimes A \vee x\in A\divideontimes\emptyset \}.$$
Now by the corollary \ref{C DV}, we will have that
$$\emptyset\oplus A =\{x: x\in\emptyset \vee x\in A\},$$
$$A\oplus\emptyset=\{x: x\in A\},$$
from this it is clear that $A\subset \emptyset\oplus A$ and that $\emptyset\oplus A\subset A$, in this way $\emptyset\oplus A=A$.
\end{proof}

\begin{theorem}\label{T fusion-union}
    Let $X$ be a $\sigma$-set, then for all $A,B\in 2^{X}$, we have that:
    $$A\oplus B = A\cup B,$$
    where $A\cup B=\{x: x\in A \vee x\in B\}.$
\end{theorem}
\begin{proof}
Let $X$ be a $\sigma$-set and $A,B\in 2^{X}$. Then, by theorem 3.39 of \cite{Gatica} we have that
$$ A\widehat{\cap} B = B\widehat{\cap} A =\emptyset , $$
in this way
$$A\divideontimes B=A \wedge B\divideontimes A=B. $$
Finally
$A\oplus B=\{x: x\in A \vee x\in B\}=A\cup B.$
\end{proof}

\begin{example}
    Let $X=\{1,2,3\}$, $A=\{1,2\}$ and $B=\{2,3\}$, it is clear that $A,B\in 2^{X}$. Now we apply the fusion operator $\oplus$.
    $$A\oplus B=\{x: x\in A\divideontimes B \vee x\in B\divideontimes A\},$$
    $$A\oplus B=\{x: x\in A\vee x\in B\},$$
    $$A\oplus B=A\cup B=\{1,2,3\}.$$
\end{example}

\begin{corollary}\label{C fusion-union}
     Let $X$ be a $\sigma$-set, then for all $A\in 2^{X}$, we have that:
    $$A\oplus X = X.$$
\end{corollary}

\begin{proof}
    Let $X$ be a $\sigma$-set and $A\in 2^{X}$. Then by theorem \ref{T fusion-union} we have that  
    $$A\oplus X = A\cup X.$$
    Now as $A\subset X$, then $A\cup X=X$, therefore 
    $$A\oplus X=X.$$
\end{proof}

\begin{example}
    Let $X=\{1,2,3,4\}$ and $A=\{1,2,3\}$, it is clear that $A\in 2^{X}$. Now we apply the fusion operator $\oplus$.
    $$A\oplus X=\{x: x\in A\divideontimes X \vee x\in X\divideontimes A\},$$
    $$A\oplus B=\{x: x\in A\vee x\in X\},$$
    $$A\oplus X=A\cup X=\{1,2,3,4\}=X.$$
\end{example}

As we said before, the fusion of $\sigma$-sets $\oplus$ is commutative by definition but as we demonstrated in \cite{Gatica,Bustamante16,Bustamante11} this operation is not associative.

\begin{example}\label{Ej no asociativo}
    Let $A=\{1^{\ast},2^{\ast}\}$, $B=\{1,2\}$ y $C=\{1\}$, then 
    $$(A\oplus B)\oplus C= \emptyset\oplus C=C  $$
and
    $$A\oplus (B\oplus C)= A\oplus B=\emptyset,  $$
therefore we have that  
$$(A\oplus B)\oplus C\neq A\oplus(B\oplus C).$$
\end{example}

\section{Generated space}

As we have already indicated in the definition \ref{Df CG} we will have that the space generated by two $\sigma$-sets $A$ and $B$ is: $$ \left\langle 2^{A},2^{B}\right\rangle=\{x\oplus y : x\in 2^{A}\wedge y\in 2^{B}\}.$$

Now taking into account the duality $\sigma$-set, $\sigma$-antiset we could consider the following example.

\begin{example}

We consider the $\sigma$-set $A=\{1,2,3\}$ and its $\sigma$-antiset $A^{-}=\{1^{\ast},2^{\ast},3^{\ast}\}$ then we obtain the integer space $3^{A}$ where,
$$3^{A}= \left\langle 2^{A}, 2^{A^{-}} \right\rangle.$$

Is important to observe that   
$$2^{A}=\{\emptyset, \{1\},\{2\},\{3\},\{1,2\},\{1,2\},\{2,3\},A \}$$
and
$$2^{A^{-}}=\{\emptyset^{-}, \{1^{\ast}\},\{2^{\ast}\},\{3^{\ast}\},\{1^{\ast},2^{\ast}\},\{1^{\ast},2^{\ast}\},\{2^{\ast},3^{\ast}\},A^{-} \}.$$

Also is important to observe that $\emptyset=\emptyset^{-}$, which is very important for the construction of $3^{A}$.

Now considering the definition of generated space,
$$3^{A}=\left\langle 2^{A}, 2^{A^{-}} \right\rangle=\{X\oplus Y : X\in 2^{A} \wedge Y \in 2^{A^{-}} \},$$
where the operator $\oplus$ is the fusion of $\sigma$-sets, we will obtain the following matrix:

\begin{table}
\centering
\begin{tabular}{||c||c|c|c|c|c|c|c|c||}
\hline
\hline
$\oplus$ & $\emptyset$ & $\{1\}$ & $\{2\}$ & $\{3\}$ & $\{1,2\}$ & $\{1,3\}$ & $\{2,3\}$ & $A$ \\\hline\hline

$\emptyset^{-}$ & $\emptyset^{0}_{0}$ & {\color{green}$\{1\}$} & {\color{green}$\{2\}$} & {\color{green}$\{3\}$} & {\color{red}$\{1,2\}$} & {\color{red}$\{1,3\}$} & {\color{red}$\{2,3\}$} & $A$  \\\hline

$\{1^{\ast}\}$     & {\color{green}$\{1^{\ast}\}$} & $\emptyset^{1}_{1}$ & {\color{cyan}$\{1^{\ast},2\}$} & {\color{cyan}$\{1^{\ast},3\}$} & {\color{green}$\{2\}$} & {\color{green}$\{3\}$} & {\color{blue}$\{1^{\ast},2,3\}$} & {\color{red}$\{2,3\}$} \\\hline

$\{2^{\ast}\}$     & {\color{green}$\{2^{\ast}\}$} & {\color{cyan}$\{1,2^{\ast}\}$} & $\emptyset^{2}_{1}$ & {\color{cyan}$\{2^{\ast},3\}$} & {\color{green}$\{1\}$} & {\color{yellow}$\{1,2^{\ast},3\}$} & {\color{green}$\{3\}$} & {\color{red}$\{1,3\}$} \\\hline

$\{3^{\ast}\}$     & {\color{green}$\{3^{\ast}\}$} & {\color{cyan}$\{1,3^{\ast}\}$} & {\color{cyan}$\{2,3^{\ast}\}$} & $\emptyset^{3}_{1}$ & {\color{magenta}$\{1,2,3^{\ast}\}$} & {\color{green}$\{1\}$} & {\color{green}$\{2\}$} & {\color{red}$\{1,2\}$} \\\hline

$\{1^{\ast},2^{\ast}\}$ & {\color{red}$\{1^{\ast},2^{\ast}\}$} & {\color{green}$\{2^{\ast}\}$} & {\color{green}$\{1^{\ast}\}$} & {\color{magenta}$\{1^{\ast},2^{\ast},3\}$} & $\emptyset^{4}_{2}$ &  {\color{cyan}$\{2^{\ast},3\}$} & {\color{cyan}$\{1^{\ast},3\}$} & {\color{green}$\{3\}$} \\\hline

$\{1^{\ast},3^{\ast}\}$ & {\color{red}$\{1^{\ast},3^{\ast}\}$} & {\color{green}$\{3^{\ast}\}$} & {\color{yellow}$\{1^{\ast},2,3^{\ast}\}$} & {\color{green}$\{1^{\ast}\}$} & {\color{cyan}$\{2,3^{\ast}\}$} & $\emptyset^{5}_{2}$ & {\color{cyan}$\{1^{\ast},3\}$} & {\color{green}$\{2\}$} \\\hline

$\{2^{\ast},3^{\ast}\}$ & {\color{red}$\{2^{\ast},3^{\ast}\}$} & {\color{blue}$\{1,2^{\ast},3^{\ast}\}$} & {\color{green}$\{3^{\ast}\}$} & {\color{green}$\{2^{\ast}\}$} & {\color{cyan}$\{1,3^{\ast}\}$} & {\color{cyan}$\{1,2^{\ast}\}$} & $\emptyset^{6}_{2}$ & {\color{green}$\{1\}$} \\\hline

$A^{-}$ & $A^{-}$ & {\color{red}$\{2^{\ast},3^{\ast}\}$} & {\color{red}$\{1^{\ast},3^{\ast}\}$} & {\color{red}$\{1^{\ast},2^{\ast}\}$} & {\color{green}$\{3^{\ast}\}$} & {\color{green}$\{2^{\ast}\}$} & {\color{green}$\{1^{\ast}\}$} & $\emptyset^{7}_{3}$ \\\hline\hline
\end{tabular}
\caption{\label{tab:matrix1}Integer Space.}
\end{table}

\end{example}

It is important to note that from the perspective of $\sigma$-sets we have that $\emptyset=\emptyset^{-}=\emptyset^{i}_{j}$ with $i\in\{0,1 ,2,3,4,5,6,7\}$ and $j\in\{0,1,2,3 \}$, where the difference of the $\sigma$-emptysets $\emptyset^{i}_{j}$ is given by annihilation, which comes from equation $A\oplus A^{-}=\emptyset$.

From the matrix representation of the integer space $3^{A}$, we can present another representation of the same integer space. This representation of the integer space $3^{A}$ is a graphical representation which we show in figure \ref{fig:Integer1}.

\begin{figure}
\centering
\includegraphics[width=0.6\linewidth]{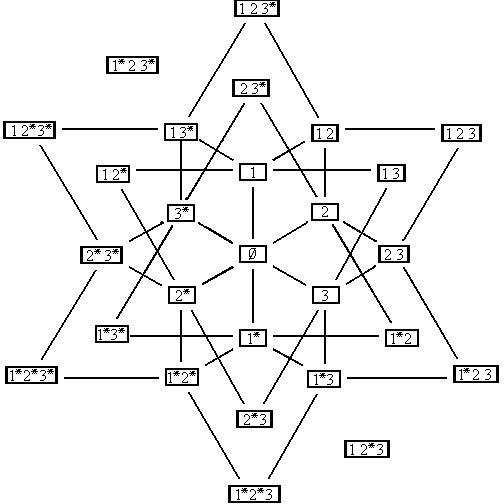}
\caption{\label{fig:Integer1} Integer Space $3^{A}$.}
\end{figure}

Finally, as a theoretical result, we have a cardinal theorem: 

\begin{theorem}
Let $A=\{1,2,3\}$, then $\left|3^{A}\right|=\left|\left\langle 2^{A}, 2^{A^{-}}\right\rangle\right|=3^{3}=27$.
\end{theorem}
\begin{proof}
    Let $A=\{1,2,3\}$, the proof is the same fusion matrix for this $\sigma$-set.
\end{proof}

We should also note that we have obtained other cardinal results for the integer space $3^{A}$ with $\left|A \right|\in\{0,1,2,3,4,5\}$. The cardinal results are as follows:  

\begin{center}
\begin{tabular}{|c|c|c|c|}
\hline
$\sigma$-Set    &  $\sigma$-Antiset  & Generated & Cardinal \\
\hline
\hline
$A=\emptyset$     & $A^{-}=\emptyset^{-}$                                 & $\left\langle 2^{A},2^{A^{-}} \right\rangle$  & $3^{0}=1$  \\
\hline
$A=\{1\}$         & $A^{-}=\{1^{\ast}\}$                                     & $\left\langle 2^{A},2^{A^{-}} \right\rangle$  & $3^{1}=3$  \\
\hline
$A=\{1,2\}$       & $A^{-}=\{1^{\ast},2^{\ast}\}$                            & $\left\langle 2^{A},2^{A^{-}} \right\rangle$  & $3^{2}=9$   \\
\hline
$A=\{1,2,3\}$     & $A^{-}=\{1^{\ast},2^{\ast},3^{\ast}\}$                   & $\left\langle 2^{A},2^{A^{-}} \right\rangle$  & $3^{3}=27$  \\
\hline                   
$A=\{1,2,3,4\}$   & $A^{-}=\{1^{\ast},2^{\ast},3^{\ast},4^{\ast}\}$          & $\left\langle 2^{A},2^{A^{-}} \right\rangle$  & $3^{4}=81$  \\ 
\hline
$A=\{1,2,3,4,5\}$ & $A^{-}=\{1^{\ast},2^{\ast},3^{\ast},4^{\ast},5^{\ast}\}$ & $\left\langle 2^{A},2^{A^{-}} \right\rangle$  & $3^{5}=243$ \\ 
\hline
\end{tabular}
\end{center}

From these calculations made with the fusion matrix we can obtain the following conjecture.

\begin{conjecture}
   Let $A$ be a $\sigma$-set such that $|A|=n$, then $\left|3^{A}\right|=\left|\left\langle 2^{A}, 2^{A^{-}}\right\rangle\right|=3^{n}$. 
\end{conjecture}

On the other hand, as we have already said, we are going to change the notation of $1_{\Theta}$ to $1_{0}$, in this way we will have the $\sigma$-set of $0$-natural numbers defined as follows:
\\
$1_{0}=\{\emptyset\}$\\
$2_{0}=\{\emptyset,1_{0}\}$\\
$3_{0}=\{\emptyset,1_{0},2_{0}\}$\\
$4_{0}=\{\emptyset,1_{0},2_{0},3_{0}\}$\\
and so on, forming the $0$-natural numbers
$$\mathds{N}^{0}=\{1_{0},2_{0},3_{0},4_{0},5_{0},6_{0},7_{0},8_{0},9_{0},10_{0},\ldots\},$$

where one of the important properties of this $\sigma$-set is that it does not annihilate with the natural numbers $\mathds{N}$ nor with the antinatural numbers $\mathds{N}^{-}$, in this way we can consider the following example for the generated space.

\begin{example}
    We consider the $\sigma$-sets $A=\{1_{0},2_{0}\}$ and $B=\{1,2\}$, therefore the space generated by $A\oplus B$ and $A\oplus B^{-}$ will be:

$\left\langle 2^{A\oplus B}, 2^{A\oplus B^{-}}\right\rangle=\{ x\oplus y : x\in 2^{A\oplus B}\wedge y\in 2^{A\oplus B^{-}}\} $

$\left\langle 2^{A\oplus B}, 2^{A\oplus B^{-}}\right\rangle=\{\emptyset, \{1_{0}\}, \{1\}, \{1^{\ast}\}, \{2_{0}\}, \{2\}, \{2^{\ast}\}, 
\{1_{0},2_{0}\}, \{1_{0},1\}, \{1_{0}, 1^{\ast}\},  \{1_{0},2\}, \{1_{0},2^{\ast}\}, \\ \{2_{0},1\}, \{2_{0}, 1^{\ast}\},  \{2_{0},2\}, \{2_{0},2^{\ast}\}, \{1,2\}, \{1, 2^{\ast}\},  \{1^{\ast},2\}, \{1^{\ast},2^{\ast}\},
\{1_{0},1,2\}, \{1_{0},1, 2^{\ast}\}, \{1_{0},1^{\ast},2\},\\ \{1_{0},1^{\ast},2^{\ast}\}, \{2_{0},1,2\}, \{2_{0},1, 2^{\ast}\}, \{2_{0},1^{\ast},2\}, \{2_{0},1^{\ast},2^{\ast}\}, \{1_{0},2_{0},1\}, \{1_{0},2_{0},1^{\ast}\}, \{1_{0},2_{0},2\}, \{1_{0},2_{0},2^{\ast}\},\\
 \{1_{0},2_{0},1,2\}, \{1_{0},2_{0},1, 2^{\ast}\}, \{1_{0},2_{0},1^{\ast},2\}, \{1_{0},2_{0},1^{\ast},2^{\ast}\}     \}$\\
 \\
In this case, the generated space becomes a meta-space generated by $A=\{1_{0},2_{0}\}$ and $B=\{1,2\}$ which can be ordered graphically as shown in figure \ref{fig:metaspace}.

\begin{figure}
\centering
\includegraphics[width=0.7\linewidth]{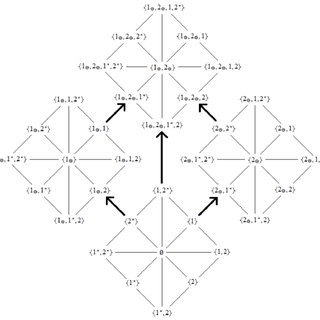}
\caption{\label{fig:metaspace} Meta-space $\left\langle 2^{A\oplus B}, 2^{A\oplus B^{-}}\right\rangle$.}
\end{figure}
    
\end{example}

Now, if we count the number of elements that the meta-space generated by $A=\{1_{0},2_{0}\}$ and $B=\{1,2\}$ has, we will find that they are 36, where the prime decomposition of this number is $36=2^{2}\cdot 3^{2}$ which is equivalent to the following multiplication of cardinals $36=2^{|A|}\cdot 3^{|B|}$, from where we can obtain the following conjecture:

\begin{conjecture}
    For all $A\in 2^{\mathds{N}^{0}}$ and $B\in 2^{\mathds{N}}$, then $\left| \left\langle 2^{A\oplus B}, 2^{A\oplus B^{-}}\right\rangle \right| = 2^{|A|}\cdot 3^{|B|} .$
\end{conjecture}

\begin{example}
    We consider $A=\{1_{0}\}$ and $B=\{1,2\}$, then we obtain that\\ 
    $\left\langle 2^{A\oplus B}, 2^{A\oplus B^{-}}\right\rangle=\{\emptyset, \{1_{0}\}, \{1\}, \{1^{\ast}\}, \{2\}, \{2^{\ast}\}, \{1_{0},1\}, \{1_{0},2\}, \{1_{0},1^{\ast}\}, \{1_{0},2^{\ast}\}, \{1,2\}, \{1,2^{\ast}\}, \\ 
     \{1^{\ast},2\}, \{1^{\ast},2^{\ast}\}, \{1_{0},1,2\}, \{1_{0},1,2^{\ast}\}, \{1_{0},1^{\ast},2\}, \{1_{0},1^{\ast},2^{\ast}\} \} $\\
Thus, we have that $|A|=1$ and $|B|=2$ and $\left| \left\langle 2^{A\oplus B}, 2^{A\oplus B^{-}}\right\rangle \right| = 2^{|A|}\cdot 3^{|B|}= 2^{1}\cdot 3^{2}=18.$
     
\end{example}

\begin{example}
    We consider $A=\emptyset$ and $B=\{1,2\}$, then we obtain that\\ 
    $3^{B} =\{\emptyset, \{1\}, \{1^{\ast}\}, \{2\}, \{2^{\ast}\}, \{1,2\}, \{1,2^{\ast}\}, \{1^{\ast},2\}, \{1^{\ast},2^{\ast}\} \} $\\
Thus, we have that $|A|=0$ and $|B|=2$ and $\left| \left\langle 2^{A\oplus B}, 2^{A\oplus B^{-}}\right\rangle \right| = 2^{|A|}\cdot 3^{|B|}= 2^{0}\cdot 3^{2}=9.$
     
\end{example}

\section{Algebraic structure of integer space $3^{A}$}
 With respect to the algebraic structure of the Integer Space $3^{A}$ for all $A\in 2^{\mathds{N}}$ we think that these structures are related with structures called NAFIL (non-associative finite invertible loops)

 \begin{theorem}\label{T Alg}
     Let $A=\{1,2\}$, then $(3^{A},\oplus)$ satisfies the following conditions:
\begin{enumerate}
    \item $(\forall X,Y\in 3^{A})(X\oplus Y\in 3^{A}),$
    \item $(\exists! \emptyset\in 3^{A})(\forall X\in 3^{A})(X\oplus \emptyset=\emptyset\oplus X=X),$
    \item $(\forall X\in 3^{A})(\exists! X^{-}\in 3^{A})(X\oplus X^{-}=X^{-}\oplus X=\emptyset),$
    \item $(\forall X,Y\in 3^{A})(X\oplus Y=Y\oplus X).$
\end{enumerate}
\end{theorem}

 \begin{proof}
     Let $A=\{1,2\}$, then we quote the fusion matrix represented in table \ref{tab:matrix2} for $3^{\{1,2\}}$.

From here it is clearly seen that conditions $(1)$, $(2)$, and $(3)$ of theorem \ref{T Alg} are satisfied, where the condition $(4)$ is obvious by definition.

We must clarify that since $\sigma$-set $\emptyset=\emptyset^{-}$, and also $\emptyset=\emptyset^{0}_{0}=\emptyset^{1}_{1}=\emptyset^{2}_{1}=\emptyset^{3}_{2}$, from here we have condition $(2)$ and the difference is in another dimension, the dimension of annihilation. Here we must clarify that the fusion operation $\oplus$ is not associative. Let $X=\{1^{\ast},2^{\ast}\}$, $Y=\{1,2\}$ and $Z=\{1\}$ then we will have that
     $(\{1^{\ast},2^{\ast}\}\oplus\{1,2\})\oplus\{1\}=\emptyset\oplus\{1\}=\{1\}$
     
on the other hand

    $\{1^{\ast},2^{\ast}\}\oplus (\{1,2\}\oplus\{1\})=\{1^{\ast},2^{\ast}\}\oplus\{1,2\}=\emptyset$

therefore we have that 

$$(X\oplus Y)\oplus Z\neq X\oplus (Y\oplus Z), $$

which shows that the structure $(3^{A},\oplus)$, is non-associative.
    
\begin{table}
\centering
\begin{tabular}{||c||c|c|c|c||}
\hline
  $\oplus$      & $\emptyset$  & $\{1\}$ & $\{2\}$ & $\{1,2\}$  \\
\hline
\hline
   $\emptyset^{-}$ & $\emptyset^{0}_{0}$  &  $\{1\}$ & $\{2\}$ & $\{1,2\}$  \\
\hline
   $\{1^{\ast}\}$  & $\{1^{\ast}\}$  &  $\emptyset^{1}_{1}$ & $\{1^{\ast},2\}$   &  $\{2\}$  \\
\hline
   $\{2^{\ast}\}$  & $\{2^{\ast}\}$  & $\{1,2^{\ast}\}$ & $\emptyset^{2}_{1}$   & $\{1\}$   \\
\hline
   $\{1^{\ast},2^{\ast}\}$  & $\{1^{\ast},2^{\ast}\}$  &  $\{2^{\ast}\}$ &  $\{1^{\ast}\}$ &  $\emptyset^{3}_{2}$   \\
\hline  
\end{tabular}
\caption{\label{tab:matrix2}Integer Space $3^{\{1,2\}}$.}
\end{table}

 \end{proof}

We now present a new conjecture.

\begin{conjecture}\label{CJ BUCLE}
     Let $A\in 2^{\mathds{N}}$, then $(3^{A},\oplus)$ satisfies the following conditions:
\begin{enumerate}
    \item $(\forall X,Y\in 3^{A})(X\oplus Y\in 3^{A}),$
    \item $(\exists! \emptyset\in 3^{A})(\forall X\in 3^{A})(X\oplus \emptyset=\emptyset\oplus X=X),$
    \item $(\forall X\in 3^{A})(\exists! X^{-}\in 3^{A})(X\oplus X^{-}=X^{-}\oplus X=\emptyset),$
    \item $(\forall X,Y\in 3^{A})(X\oplus Y=Y\oplus X).$
\end{enumerate}
\end{conjecture}

\section{$\sigma$-Sets Equations}

Continuing with the analysis of the $\sigma$-sets, we now have the development of the equations of $\sigma$-sets of a $\sigma$-set variable, equations that play a very important role when solving a $\sigma$-set equation, now let's define and go deeper into the $\sigma$-sets variables.

We must remember that for every $\sigma$-set $A$ and $B$, the fusion of both is defined as:

$$A\oplus B=\{x: x\in A\divideontimes B \vee x\in B\divideontimes A  \}$$
 
\begin{definition}
 Let $A$ be a $\sigma$-set, then $A$ is said to be an entire $\sigma$-set if there exists the $\sigma$-antiset $A^{-}$.   
\end{definition}

\begin{example}
    Let $A=\{1_{0},2_{0},3_{0}\}$, then this $\sigma$-set is not an integer, since $A^{-}$ does not exist, on the other hand the $\sigma$-set $A=\{1,2,3,4\}$, is an integer $\sigma$-set since $A^{-}=\{1^{\ast},2^{\ast},3^{\ast},4^{\ast}\}$ exists which is the $\sigma$-antiset of $A$.
\end{example}

It is clear that if a $\sigma$-set $A$ is integer, then by definition there exists the integer space $3^{A}$. We should also note that if $A$ is an integer $\sigma$-set, then $\left[A\cup A^{-}\right]$ is a proper $\sigma$-class, for example, consider $A=\{1,2\}$, then $\left[A\cup A^{-}\right]=\left[1,2,1^{\ast},2^{\ast}\right]$, is a proper $\sigma$-class.
We must observe that $\sigma$-set theory \cite{Gatica} is a theory of $\sigma$-classes, where $\sigma$-sets are characterized by axioms. We must also note that a proper $\sigma$-class is a $\sigma$-class that is not a $\sigma$-set. This difference is essential to give rise to the existence of antielements along with their respective elements.

\begin{definition}
    Let $A$ be a integer $\sigma$-set such that $|A|=n$, then $X$ is said to be a $\sigma$-set variable of $3^{A}$, if and only if
$$X=\{x_{1},x_{2},x_{3},\ldots ,x_{m}\},$$
where $m\leq n$ and $x_{i}$ a variable of the proper class $\left[A\cup A^{-}\right]$.
\end{definition}

\begin{example}
    Let $A=\{1,2,3\}$ be a $\sigma$-set, it is clear that $A$ is an entire $\sigma$-set since there exists $A^{-}=\{1^{\ast},2^{\ast},3^{\ast}\}$ and therefore $3^{A}$, in this way we will have that 

$$X=\emptyset,$$    
$$X=\{x\},$$
$$X=\{x_{1},x_{2}\},$$
$$X=\{x_{1},x_{2},x_{3}\},$$
are $\sigma$-sets variables of $3^{A}$, where $x,x_{1},x_{2},x
_{3}\in \left[1,2,3,1^{\ast},2^{\ast},3^{\ast}\right]$.
\end{example}

\begin{lemma}\label{L VC}
Let $A$ be an integer $\sigma$-set and $X$ a $\sigma$-set variable of $3^{A}$, then $A\oplus X=A\cup X$, with $A\subset A\cup X$ and $X\subset A\cup X$.
\end{lemma}

\begin{proof}
   Let $A$ be an integer $\sigma$-set and $X$ a $\sigma$-set variable of $3^{A}$, then
$$A\oplus X=\{x: x\in A\divideontimes X \vee x\in X\divideontimes A \} $$
Now we have that
$$A\divideontimes X=A$$
and
$$X\divideontimes A=X$$
since $X$ is a $\sigma$-set variable, therefore we will have that
$$A\oplus X=\{x: x\in A \vee x\in X \}=A\cup X.$$
We can also observe that $A\cap X=\emptyset$ since $X$ is a $\sigma$-set variable, therefore $A\subset A\cup X$ and $X\subset A\cup X$.
\end{proof}

\begin{example}
    Let $A=\{1,2,3\}$, and $X$ be a $\sigma$-set variable of $3^{A}$, that is,
$$X=\emptyset$$
$$X=\{x\},$$
$$X=\{x_{1},x_{2}\},$$
$$X=\{x_{1},x_{2},x_{3}\},$$
are $\sigma$-sets variable of $3^{A}$, where $x,x_{1},x_{2},x
_{3}\in \left[1,2,3,1^{\ast},2^{\ast},3^{\ast}\right]$.
then
    $$A\oplus X=\{1,2,3\},$$
    $$A\oplus X=\{1,2,3,x\},$$
    $$A\oplus X=\{1,2,3,x_{1},x_{2}\},$$
    $$A\oplus X=\{1,2,3,x_{1},x_{2},x_{3}\}$$
\end{example}

After the lemma \ref{L VC} we proceed to analyze some equations of a $\sigma$-set variable and their solutions

Let $A$ be an integer $\sigma$-set, $X$ a $\sigma$-set variable and $M$ and $N$ two $\sigma$-sets of the integer space $3^{A}$, then an equation of a $\sigma$-set variable will be

$$X\oplus M=N.$$

 Now if $M=N$, then the equation becomes
$$X\oplus M=M,$$
and by the corollary \ref{C fusion-union} we have that the solutions are all $X\in 2^{M}$, where we naturally count $X=\emptyset$, hence we have an equation of a $\sigma$-set variable with multiple solutions.

Now consider $M\neq N$, then the $\sigma$-set equation becomes:

 $$X\oplus M=N,$$

 We must remember that the structure in general is not associative, therefore we cannot freely use this property, so to find the solution to the equation we must develop a previous theorem. To develop this theorem we will assume that for every integer $\sigma$-set $A$ the generated space is $\left\langle 2^{A}, 2^{A^{-}} \right\rangle=3^{A}$, and also that $3^{A}$ satisfies conjecture \ref{CJ BUCLE}.  
 
\begin{theorem}\label{T IG}
    Let $A$ be an integer $\sigma$-set, $X$ be a $\sigma$-set variable of $3^{A}$ and $M\in 3^{A}$. Then
$$(X\oplus M)\oplus M^{-}=X$$
\end{theorem}

\begin{proof}
 Let $A$ be an integer $\sigma$-set, $X$ be a $\sigma$-set variable of $3^{A}$ and $M\in 3^{A}$, then by lemma \ref{L VC} we have that $X\oplus M=X\cup M$, with $X\cap M=\emptyset$.

 Therefore we have that
 $$ \circledast \ (X\oplus M)\oplus M^{-}=\{a: a\in(X\oplus M)\divideontimes M^{-}\vee a\in M^{-}\divideontimes (X\oplus M)\}$$
 $$ =\{a: a\in(X\cup M)\divideontimes M^{-}\vee a\in M^{-}\divideontimes (X\cup M)\} $$
 so
 
 $(X\cup M)\divideontimes M^{-}= (X\cup M)-(X\cup M)\widehat{\cap} M^{-}=(X\cup M)-M=X,$
 
 and

 $M^{-}\divideontimes(X\cup M)= M^{-}-M^{-}\widehat{\cap} (X\cup M)=M^{-}-M^{-}=\emptyset.$

Now replacing these calculations in $(\circledast)$ we will have that

$$(X\oplus M)\oplus M^{-}=\{a: a\in X \vee a\in\emptyset\}$$
$$(X\oplus M)\oplus M^{-}=\{a: a\in X\},$$
$$(X\oplus M)\oplus M^{-}=X.$$
\end{proof}

Now, after theorem \ref{T IG} has been proved, we can solve some $\sigma$-set equation for the integer $\sigma$-set $A=\{1,2\}$, since the generated space is effectively equal to $3^{A}$, that is, $\left\langle 2^{A}, 2^{A^{-}} \right\rangle=3^{A}$, and also $3^{A}$ is a non-associative abelian loop.

Let $A=\{1,2\}$ be an integer set and $M,N\in 3^{A}$, with $M\widehat{\cap} N=\emptyset$, then the equation
 $$X\oplus M=N,$$
has the following solution
$$X\oplus M=N \setminus \oplus M^{-},$$
$$(X\oplus M)\oplus M^{-}=N\oplus M^{-},$$

then by theorem \ref{T IG} we will have that 

$$X=N\oplus M^{-}.$$

Let us now show a concrete example for $A=\{1,2\}$.

\begin{example}
    Let $A=\{1,2\}$ be an integer $\sigma$-set, $M=\{1,2^{\ast}\}$ and $N=\{1\}$, with $M\widehat{\cap}N=\emptyset$, then the equation of a $\sigma$-set variable 
    $$X\oplus \{1,2^{\ast}\}=\{1\}$$
has the following solution.
$$X\oplus \{1,2^{\ast}\}=\{1\} \setminus \oplus \{1^{\ast},2\},$$
$$(X\oplus \{1,2^{\ast}\})\oplus \{1^{\ast},2\}=\{1\}\oplus \{1^{\ast},2\},$$
$$X=\{2\}.$$

Here we can see that the equation has as solution the $\sigma$-set $S_{1}=\{2\}$, since
$$\{2\}\oplus \{1,2^{\ast}\}=\{1\},$$
but like the equation $X\oplus M=M$, this one does not have a unique solution since the $\sigma$-set $S_{2}=\{1,2\}$, is also a solution for the equation of a $\sigma$-set variable,
$$\{1,2\}\oplus\{1,2^{\ast}\}=\{1\}.$$

In this way we have two solutions for our equation of a $\sigma$-set variable which are:
$$ S =\{S_{1},S_{2}\}=\{ \{2\}, \{1,2\}\}.$$
\end{example}

Note that if $M\widehat{\cap}N=\emptyset$ then the $\sigma$-set equation has a solution, but otherwise the $\sigma$-set equation has an empty solution.

\begin{example}
     Let $A=\{1,2\}$ be an integer $\sigma$-set, $M=\{1^{\ast}\}$ and $N=\{1\}$, with $M\widehat{\cap}N=\{1^{*}\}$, then the equation of a $\sigma$-set variable 
    $$X\oplus \{1^{\ast}\}=\{1\}$$
There is no solution.
$$X\oplus \{1^{\ast}\}=\{1\} \setminus \oplus \{1\},$$
$$(X\oplus \{1^{\ast}\})\oplus \{1\}=\{1\}\oplus \{1\},$$
$$X=\{1\},$$
which is a contradiction, because
$$\{1\}\oplus \{1^{\ast}\}=\{1\},$$
$$\emptyset=\{1\}.$$

\end{example}

\begin{definition} \label{fusionableDef}
    A \(\sigma\)-set equation \(X \oplus M=N\) is said to be {\bf fusionable} if \(M \widehat{\cap} N=\emptyset\).
\end{definition}

With this in mind, let us conclude with a bounded theorem to find some solutions of the \(\sigma\)-set equation.

\begin{theorem} \label{theoremSolutionsS1S2}
    Let $A$ be an integer $\sigma$-set, $X$ a $\sigma$-set variable of $3^{A}$, and $M,N\in 3^{A}$, then two possible solutions \(S=\{S_1,S_2\}\) of the fusionable equation
$$X\oplus M=N,$$
are  $S_1=N \oplus R^{-}$ and $S_2=R^{-}$, where $R:=M \oplus N^{-}.$
\end{theorem}

\begin{proof}
    For the first solution \(S_1\) we have that 
\begin{eqnarray*}
	S_1 &=& N \oplus R^{-} \\
	    &=& N \oplus (M \oplus N^{-})^{-} \\
	    &=& N \oplus (N \oplus M^{-}) \\
	    &=& N \oplus M^{-},
\end{eqnarray*}
where \(S_2=(M \oplus N^{-})^{-}=N \oplus M^{-}\) because of the result iteration seen above. Hence both results are actually a fusion solution for \(X \oplus M=N,\) where \(S_2=R^{-}\) is an exact solution and \(S_1=N \oplus R^{-}\) is an intersected rest solution. Because of \(M \widehat{\cap}N=\emptyset\) (Definition \ref{fusionableDef}) as the equation \(X \oplus M=N\) is fusionable, both \(S_1 \oplus M\) and \(S_2 \oplus M\) will be fusionable into another \(\sigma\)-set \(N\).
\end{proof}

As we looked above, the solution space is reduced such that the solutions are indeed \(N \oplus M^{-},\) being by consequence possible solutions for the fusionable equation \(X \oplus M=N\).

\begin{example}
    Let \(A=\{1,2,3,4,5,6\}\) be an integer $\sigma$-set, $M=\{1,2,3^{*},4^{*},5,6^{*}\}$ and $N=\{1,2\}$, then the equation of a $\sigma$-set variable
\[
 X \oplus \{1,2,3^{*},4^{*},5,6^{*}\}=\{1,2\},
\]
which is fusionable because \(M \widehat{\cap}N=\{1,2,3^{*},4^{*},5,6^{*}\} \widehat{\cap}\{1,2\}=\emptyset.\) 

Now, by using Theorem \ref{theoremSolutionsS1S2}, let us first obtain
\begin{eqnarray*}
	R^{-} &=& (M \oplus N^{-})^{-} \\
	  &=& (\{1,2,3^{*},4^{*},5,6^{*}\} \oplus \{1,2\}^{-})^{-} \\
        &=& (\{1,2,3^{*},4^{*},5,6^{*}\} \oplus \{1^{*},2^{*}\})^{-} \\
	    &=& (\{3^{*},4^{*},5,6^{*}\})^{-} \\
	  &=& \{3,4,5^{*},6\},
\end{eqnarray*}
so we get \(S_1=N \oplus R^{-}=\{1,2,3,4,5^{*},6\}\) and \(S_2=R^{-}=\{3,4,5^{*},6\}\), which can be easily proved that both solutions gives \(S_1 \oplus M=S_2 \oplus M=N\) as a resulting \(\sigma\)-set. Hence \(S=\{\{1,2,3,4,5^{*},6\},\{3,4,5^{*},6\}\}\) is a solution set for the fusionable equation \(X \oplus M=N\).
\end{example}

\section{Conclusions}
One of the first conclusions we can draw is that the fusion operator $\oplus$ for $\sigma$sets is equivalent to the union operator for sets within the context of the set of parts $2^{A}$, which allows us to deduce that the fusion of $\sigma$-sets is an extension of the union for the generated space.

The fact that the integer space $3^{A}$ presents a cardinal of power 3, is very important for the development of the theory of transfinite numbers, since in general the power set $2^{A}$ that goes to the power of 2 is used; in this way our results can serve as an impetus for the development of the theory of transfinite numbers.

We can also conclude that the algebraic structure of the integer space $3^{1,2}$ is a loop, which leads us to conjecture that the integer space in general has a loop structure. This fact is relevant to $\sigma$-set theory since, if it were so, it would show that the fusion operator $\oplus$ is not associative which is relevant for solving set equations.

As a final conclusion, we can state that we can generate $\sigma$-set equations given the existence of inverses for the fusion operator $\oplus$ in the integer space, but in the general case, solutions are not given, so a condition must be imposed on the $\sigma$-sets of the equation. We have not yet conducted a detailed study on the number of solutions to each set equation, leaving this study for future research.

To see more works in which antisets or $\sigma$-antiset are used or in which equation $A\cup B=\emptyset$ is described, visit the references \cite{Bustamante11}, \cite{Bustamante16}, \cite{Chunlin}, \cite{Gatica}, \cite{Sengupta}.


\end{document}